\documentclass[11pt]{article}
\usepackage{latexsym, amsmath, amsfonts, amscd, amssymb, verbatim,mathrsfs, amsthm}
\usepackage{float}
\restylefloat{table}
\usepackage{color}
\usepackage[colorinlistoftodos]{todonotes}
\usepackage{hyperref}
\hypersetup{colorlinks=true, allcolors = blue}

\usepackage[margin=1in]{geometry}  % set the margins to 1in on all sides

\usepackage{titlesec}

\titleformat*{\section}{\large\bfseries}

\newtheorem{Theorem}{Theorem}[section]
\newtheorem*{Theorem*}{Theorem}

\newtheorem{Lemma}[Theorem]{Lemma}
\theoremstyle{definition}
\newtheorem{Remark}[Theorem]{\it Remark}

\newtheorem{Algorithm}[Theorem]{Algorithm}
\newtheorem*{Algorithm*}{Algorithm}
\newtheorem{Corollary}[Theorem]{Corollary}
\newtheorem{Definition}[Theorem]{Definition}
\newtheorem{Example}[Theorem]{Example}
\newtheorem*{Proof}{Proof}
\normalsize

\makeatletter
\renewcommand\@biblabel[1]{#1.}
\makeatother

\newcommand{\R}{\mathbb{R}}
\DeclareMathOperator{\vi}{\rm{VI}}
\DeclareMathOperator{\pr}{\rm{Pr}}

\usepackage{booktabs}
\usepackage{multicol}
\usepackage{multirow}
\usepackage{graphicx}
\usepackage{array}
\newcolumntype{C}[1]{>{\centering\let\newline\\\arraybackslash\hspace{0pt}}m{#1}}

\begin{document}	
	\title{\bf \Large Variational Inequalities \\ Governed By Strongly Pseudomonotone Operators}
	
	\author{Pham Tien Kha\footnote{Department of Mathematics, Ho Chi Minh City University of Education, 280 An Duong Vuong, Ho Chi Minh, Vietnam. E-mails: phamtienkha@gmail.com; khapt@hcmue.edu.vn}, \qquad
		Pham  Duy Khanh\footnote{Department of Mathematics, HCMC University of Education, Ho Chi Minh, Vietnam and Center for Mathematical
			Modeling, Universidad de Chile, Santiago, Chile. E-mails:
			pdkhanh182@gmail.com; pdkhanh@dim.uchile.cl}
	}		
	
	\maketitle
	\date{}
	
	\begin{quote}
		\noindent {\bf Abstract.} Qualitative and quantitative aspects for variational inequalities governed by strongly pseudomonotone operators on Hilbert space are investigated in this paper. First, we establish a global error bound for the solution set of the given problem with the residual function being the normal map. Second, we will prove that the iterative sequences generated by gradient projection method (GPM) with stepsizes forming a non-summable diminishing sequence of positive real numbers converge to the unique solution of the problem when the operator is bounded over the constraint set. Two counter-examples are given to show the necessity of the boundedness assumption and the variation of stepsizes. We also analyze the convergence rate of the iterative sequences generated by this method. Finally, we give an in-depth comparison between our algorithm and a recent related algorithm through several numerical experiments.
		
		
	\end{quote}
	
	\medskip
	
	\section{Introduction}
	
	Variational inequality (VI) is a powerful mathematical model which unifies the study of important concepts such as optimization problems, equilibrium problems, complementarity problems, obstacle problems and continuum problems in the mathematical sciences (see e.g. \cite{pang, stam}).
	
	Qualitative properties of VI strongly depend on some kind of monotonicity. In particular, the existence and uniqueness of the solution to the VI can be established under strong monotonicity. In view of the natural residue of the projection, Facchinei and Pang \cite{pang} obtained an upper error bound for strongly monotone and Lipschitz continuous VI. In \cite{khanh-minh}, Khanh and Minh introduced a sharper error bound for this class of VI and gave a counter-example to show the necessity of Lipschitz continuity. Moreover, an extended result for strongly pseudomonotone and Lipschitz continuous VI was established in \cite{kim-vuong-khanh}. Such error bound not only plays an important role in proving the convergence of algorithms but also serves as a termination criteria for iterative algorithms. A question arises: can we find an error bound that does not require the Lipschitz continuity assumption? In this paper, by using the normal map which is closely related to the natural map as the residual function, we present a new error bound for strongly pseudomonotone VIs. 
	
	There are several algorithms solving VI with certain monotonicity and continuity assumptions. Among those methods, the GPM \cite[Algorithm 12.1.1]{pang} which solves strongly monotone and Lipschitz continuous VI is one of the cheapest. A modified GPM with variable stepsizes solving strongly pseudomonotone and Lipschitz continuous VI has recently been established in \cite{khanh-vuong}. They also proposed a GPM with non-summable diminishing stepsize sequence in which we do not need to know a priori constants. Following this idea, we will prove in this paper that the Lipschitz continuity can be completely omitted in the modified GPM, however the boundedness of the operator over the constraint set is required. A counter-example is given to show the necessity of this boundedness assumption. We also give a counter-example to show that the traditional GPM with constant stepsize cannot be applied when the Lipschitz continuity is omitted. When the stepsizes are sequences of terms defining the $p$-series, we can estimate the rate of convergence of modified GPM which depends on the interval containing $p$. 
	
	Following this introduction, we give some preliminaries in Section 2 in which we recall some well-known definitions and properties of the projection mapping, kinds of monotonicity as well as the natural map and the normal map. In Section 3, we establish the error bound for strongly pseudomonotone VIs. In Section 4, we recall the classical GPM and give a counter-example to show its unavailability when omitting Lipschitz continuity condition. A modification for this method is proposed for the given problem. Some convergence rate results are established in Section 5. Some numerical experiments and comparisions with related works are given in Section 6. Finally, concluding remarks are given in Section 7. 
	
	\section{Preliminaries}
	
	Consider a Hilbert space $H$ with scalar product $\langle \cdot\ , \cdot \rangle$. Let $K\subset H$ be a non-empty closed convex set and $F\colon K\to H$ be an operator. The variational inequality problem defined by $K$ and $F$, denoted by $\vi(K, F)$, is to find $x^\ast\in K$ such that 
	\begin{equation}\label{VI}
	\langle F(x^\ast), x - x^\ast \rangle \ge 0, \quad \forall x\in K.
	\end{equation}
	Clearly, if $x^\ast$ satisfies \eqref{VI} and belongs to the interior of $K$ then $F(x^\ast) = 0$.
	
	For each $x\in H$, there exists a unique point in $K$ \cite[Chapter 1, Lemma 2.1]{stam}, denoted by $\pr_K(x)$, such that $$\|x - \pr_K(x)\| \le \|x - y\|, \quad \forall y\in K. $$
	
	The point $\pr_K(x)$ is called the projection of $x$ on $K$. Some well-known properties of the projection mapping $\pr_K\colon H\to K$ are recalled in the following theorem (see \cite[Chapter 2]{ekeland} and \cite[Chapter 1, Theorem 2.3]{stam}).
	
	\begin{Theorem}\label{projection_char}
		Let $K\subset H$ be a non-empty closed convex set.
		\begin{enumerate}
			\item [{\rm(a)}] For all $x\in H$ and $y\in K$, it holds that $$\langle x - \pr_K(x), y - \pr_K(x)\rangle \le 0.$$
			
			\item [{\rm (b)}] The projection mapping is non-expansive, that is $$\|\pr_K(x) - \pr_K(y) \| \le \|x - y\|, \quad \forall x, y \in H.$$
		\end{enumerate}
	\end{Theorem}
	
	One often considers $\vi(K, F)$ when $F$ possesses a certain monotonicity property. 
	
	\begin{Definition} (see \cite{karamardian} and \cite{kara-schai})
		Let $K\subset H$ be arbitrary. The mapping $F\colon K\to H$ is said to be
		\begin{enumerate}
			\item [(a)] monotone on $K$ if $$\langle F(x) -  F(y), x - y \rangle \ge 0, \quad \forall x, y\in K.$$
			
			\item [(b)] strongly monotone on $K$ if there exists $\gamma > 0$ such that $$\langle F(x) -  F(y), x - y \rangle \ge \gamma \|x - y \|^2, \quad \forall x, y\in K.$$
			
			\item [(c)] pseudomonotone on $K$ if $$\langle F(y), x - y \rangle \ge 0 \implies \langle F(x), x - y \rangle \ge 0, \quad \forall x, y\in K.$$
			
			\item [(d)] strongly pseudomonotone on $K$ if there exists $\gamma > 0$ such that $$\langle F(y), x - y \rangle \ge 0 \implies \langle F(x), x - y \rangle \ge \gamma \|x - y \|^2, \quad \forall x, y\in K.$$
		\end{enumerate}
	\end{Definition}
	Obviously, the following relations hold: ${\rm(b)}\implies{\rm(a)} \implies{\rm(c)}$ and ${\rm(b)}\implies{\rm(d)}\implies{\rm(c)}$. The reversed implications are not true in general.
	
	\begin{Remark}\text{}
		\begin{enumerate}
			\item If $F$ is strongly monotone or strongly pseudomononotone on $K$, $\vi(K, F)$ has at most one solution.
			\item When $F$ is continuous on finite dimensional subspaces of $H$ and strongly pseudomonotone on $K$, $\vi(K, F)$ has a unique solution \cite[Theorem 2.1]{kim-vuong-khanh} (the mapping $F$ from $K$ to $H$ is {\it continuous on finite dimensional subspaces of $H$} if for any finite dimensional subspace $M\subset H$, the restriction of $F$ to $K\cap M$ is weakly continuous; see \cite[Chapter 3, Definition 1.2]{stam}).
		\end{enumerate}
	\end{Remark}
	
	\medskip
	\medskip
	
	We recall the Lipschitz continuity of a mapping.
	
	\begin{Definition}
		Let $K\subset H$ be arbitrary. A mapping $F\colon K\to H$ is said to be Lipschitz continuous on $K$ if there exists $L > 0$ such that $$\|F(x) - F(y)\| \le L\|x - y \|, \quad \forall x, y\in K.$$
	\end{Definition}
	
	\medskip
	
	Now we consider two well-known mappings associated with the problem VI($K,F$): the natural map $F_K^\text{nat}$ and the normal map $F_K^\text{nor}$. 
	
	\begin{Definition}
		Let $K\subset H$ be a non-empty closed convex set and $F\colon K\to H$ be arbitrary.
		\begin{enumerate}
			\item [(a)] The natural map $F_K^\text{nat}\colon K\to H$ is defined as $$F_K^\text{nat}(x) := x - \pr_K(x-F(x)), \quad \forall x\in K.$$
			
			\item [(b)] The normal map $F_K^\text{nor}\colon H\to H$ is defined as $$F_K^\text{nor}(x) := F(\pr_K(x))+x-\pr_K(x), \quad \forall x\in H.$$
		\end{enumerate}
	\end{Definition}
	The mappings $F_K^\text{nat}$ and $F_K^\text{nor}$ are  very useful for characterizing the solution set of $\vi(K, F)$ \cite[Propositions 1.5.8 and 1.5.9]{pang}. The results in \cite[Propositions 1.5.8 and 1.5.9]{pang} are proven in finite dimensional space and one can follow the same pattern to prove the following generalized version in Hilbert space. 
	
	\begin{Theorem}\label{F_K^nat}
		Let $K\subset H$ be non-empty closed convex set and $F\colon K\to H$ be arbitrary.
		
		\begin{enumerate}
			\item [\rm (a)] $x^\ast$ is a solution of $\vi(K, F)$ if and only if $F_K^\text{nat}(x^\ast) = 0$.
			
			\item [\rm (b)] $x^\ast$ is a solution of $\vi(K, F)$ if and only if there exists $z\in H$ such that $x^\ast = \pr_K(z)$ and $F_K^\text{nor}(z) = 0$.
		\end{enumerate}
	\end{Theorem}

	\section{Error bound for strongly pseudomonotone VIs}
	
	With the help of degree theory, Facchinei and Pang proved that VI associated with strongly monotone and continuous operator admits a unique solution. The following error bound (which was originally proven in finite dimensional space, but we can use the same proof for Hilbert space) is widely used in that case \cite[Theorem 2.3.3]{pang}.
	
	\begin{Theorem}\label{old_errorbound}
		Let $K\subset H$ be a non-empty closed convex set, $F\colon K\to H$ be Lipschitz continuous with constant $L$ and strongly monotone with modulus $\gamma$, and $x^\ast$ be the unique solution of $\vi(K, F)$. For all $x\in K$, we have $$\| x - x^\ast\| \le \frac{L+1}{\gamma} \|x-\pr_K(x-F(x))\|.$$
	\end{Theorem}
	
	Extending \cite[Theorem 2.3.3]{pang}, Kim et al. proved in \cite[Theorem 2.1]{kim-vuong-khanh} the solution uniqueness for strongly pseudomonotone VI. Moreover, they established an error bound for strongly pseudomonotone and Lipschitz continuous VIs \cite[Theorem 4.2]{kim-vuong-khanh}. Recently, a sharper upper error bound and a new lower error bound for strongly monotone and Lipschitz continuous VIs were establised in \cite[Theorem 3.1]{khanh-minh}. The authors also showed in \cite{khanh-minh} that we cannot omit the Lipschitz continuity assumption in Theorem \ref{old_errorbound} \cite[Remark 3.1]{khanh-minh}. To deal with the non-Lipschitz case, we could establish a new error bound by using the normal map.
	
	\begin{Theorem}\label{new_errorbound}
		Let $K\subset H$ be non-empty closed convex and $F\colon K\to H$ be strongly pseudomonotone with modulus $\gamma$. Suppose that $\vi(K, F)$ admits a unique solution $x^\ast$. For all $x\in H$, we have 
		\begin{equation}\label{new_errorbound_ineq}
		\|x^\ast - \pr_K(x)\| \le \frac{1}{\gamma}\|F_K^\text{nor}(x)\|.
		\end{equation}
	\end{Theorem}
	\begin{Proof}
		For a given vector $x\in H$, write $r = F_K^\text{nor}(x)$. By Theorem \ref{projection_char}(a), for every $y\in K$,
		$$
		\langle x - \pr_K(x), \pr_K(x) - y \rangle \ge 0.
		$$
		Substituting $\pr_K(x) = F(\pr_K(x)) + x - r$ and $y = x^\ast$ into the above inequality, we obtain 
		$$
		\langle r - F(\pr_K(x)), \pr_K(x) - x^\ast \rangle \ge 0.
		$$
		This inequality is equivalent to
		\begin{equation}
		\label{err1} \langle r, \pr_K(x) - x^\ast\rangle \ge \langle F(\pr_K(x)), \pr_K(x) - x^\ast\rangle.
		\end{equation}
		Since $x^\ast$ is the solution of $\vi(K, F)$, we have 
		$$
		\langle F(x^\ast), \pr_K(x) - x^\ast \rangle \ge 0.
		$$
		By the strong pseudomonotonicity of $F$, the right-hand side of \eqref{err1} is not smaller than $\gamma \|x^\ast-\pr_K(x)\|^2$, while the left-hand side is not greater than $\|r\|\cdot \|x^\ast - \pr_K(x)\|$ by Cauchy-Schwarz inequality. Therefore,
		$$
		\|r\| \cdot \|x^\ast - \pr_K(x)\|\ge \gamma \|x^\ast - \pr_K(x)\|^2, 
		$$
		which deduces to \eqref{new_errorbound_ineq}. \qed
	\end{Proof}

	\begin{Remark}\label{remark_new_errorbound} If $x\in K$, $\pr_K(x) = x$. Thus $F_K^\text{nor}(x) = F(\pr_K(x))+x-\pr_K(x)=F(x)$. It follows from \eqref{new_errorbound_ineq} that
		
		\begin{equation}\label{new_errorbound_xinK}
		\|x^\ast - x\| \le \frac{1}{\gamma} \|F(x)\|, \quad \forall x\in K.
		\end{equation}
		
		It deduces from \eqref{new_errorbound_xinK} that for an arbitrary $x\in K$, $x^\ast$ is always in the closed ball with center $x$ and radius $\frac{1}{\gamma}\|F(x)\|$. In case $x^\ast$ lies in the interior of $K$, $F(x^\ast) = 0$. With an additional assumption that $F$ is continuous on $K$, \eqref{new_errorbound_xinK} can be used as a stopping criterion for methods solving strongly pseudomonotone VIs.
		
	\end{Remark}

	\section{Gradient projection method for strongly pseudomonotone VIs}
	
	We recall the classical gradient projection method solving $\vi(K, F)$ where $F$ is Lipschitz continuous with constant $L$ and strongly pseudomonotone with modulus $\gamma$. It is well-known that the iterative sequences generated by this method converge to the unique solution of the given problem (see \cite[Theorem 4.1]{khanh-vuong}). 
	
	\begin{Algorithm} {\it (Gradient projection algorithm with constant stepsize)}
		\label{old_algorithm}
		\begin{quote}
			{\bf Data.} Select $x_1\in K$ and $\lambda \in \left(0, \frac{2\gamma}{L^2} \right)$.
			
			{\bf Step 0:} Set $k = 1$.
			
			{\bf Step 1:} Compute $x_{k+1} = \pr_K(x_k - \lambda F(x_k))$.
			
			{\bf Step 2:} Check $x_{k+1} = x_k$. {\bf If} Yes {\bf then} Stop. {\bf Else} set $k = k+1$ and go to {\bf Step 1}.
		\end{quote}
	\end{Algorithm}
	
	\medskip
	The following example shows that the iterative sequence may not converge to the solution when the Lipschitz continuity of $F$ is omitted and $\lambda \in (0, 1)$.
	
	\begin{Example}\label{notLip}
		Let $K = [-1, 1]$ and $F\colon K \to \R$ be defined as 
		$$
		F(x) = \begin{cases} \;\;\;2\sqrt{x} \quad &\text{if}\quad \;\;\;\;0\le x\le 1 \\ -2\sqrt{-x} \quad &\text{if}
		\quad  -1\le x< 0\end{cases}.
		$$
		Since $F- \text{id}$ is an increasing function on $K$, $F$ is strongly monotone with modulus $1$ on $K$. On the other hand, $F$ is not Lipschitz continuous on $K$ since 
		$$
		\frac{F(x)-F(0)}{x-0} \to \infty \quad \text{as} \quad x\to 0^+.
		$$
		Moreover, $\vi(K, F)$ has a unique solution $x^\ast = 0$. Let $\lambda \in (0, 1), x_1 \in (0, \lambda^2)\subset K$ and $\{x_k\}_{k\ge 1}$ be the iterative sequence generated by  Algorithm \ref{old_algorithm}. 
		Observe that for an arbitrary $k$, if $0 < x_k < \lambda^2$ then $0 < x_k < x_{k+2} < \lambda^2$. 
		Indeed, since $0 < x_k < \lambda^2 < 1$, we have 
		$$x_k - \lambda F(x_k) = x_k - 2\lambda \sqrt{x_k} \in (-\lambda^2, 0) \subset (-1, 0)\subset K,$$
		thus
		$$x_{k+1} = \pr_K(x_k - \lambda F(x_k)) = x_k - 2\lambda \sqrt{x_k} \in (-\lambda^2, 0)\subset (-1,0).$$
		Next, we have 
		$$x_{k+1} - \lambda F(x_{k+1}) = x_{k+1} + 2\lambda \sqrt{-x_{k+1}} \in (0, \lambda^2) \subset (0, 1)\subset K ,$$
		then
		$$x_{k+2} = \pr_K(x_{k+1} - \lambda F(x_{k+1})) = x_{k+1} + 2\lambda \sqrt{-x_{k+1}} \in (0, \lambda^2).$$	
		It remains to show that $x_{k+2} > x_k$. We have $$\begin{aligned} x_{k+2} - x_k &=- 2\lambda \sqrt{x_k}+ 2\lambda \sqrt{2\lambda \sqrt{x_k} - x_k} \\& = 4\lambda \sqrt{x_k} \cdot \frac{\lambda - \sqrt{x_k}}{\sqrt{x_k}+\sqrt{2\lambda \sqrt{x_k} - x_k}} > 0, \end{aligned}$$ 
		which is true since $0 < x_k < \lambda^2$. 
		Following this observation, since $0 < x_1 < \lambda^2$, it can be proved by induction that $$0 < x_{2k+1}  <x_{2k+3} < \lambda^2, \quad \forall k\ge 0,$$
		which means $\{x_{2k+1}\}_{k\ge 0}$ is an increasing positive sequence. Thus $\{x_{2k+1}\}_{k\ge 0}$ is a subsequence of $\{x_k\}_{k\ge 1}$ that does not converge to $0$ which implies $\{x_k\}_{k\ge 1}$ does not converge to $0$. \end{Example}
	
	\medskip
	
	We now consider the case $F$ is merely strongly pseudomonotone and bounded on the constraint set $K$. Clearly, if $K$ is a bounded set, Lipschitz continuity leads to the boundedness of $F$ on $K$.
	
	\begin{Algorithm} {\it(Gradient projection algorithm with variable stepsizes)}\label{new_algorithm}
		\begin{quote}
			{\bf Data.} Select $x_1\in K$ and a positive sequence of stepsizes $\{\lambda_k\}$ satisfying $\displaystyle \sum_{k = 1}^\infty \lambda_k = \infty$ and $\displaystyle \lim_{k\to \infty} \lambda_k = 0$.
			
			{\bf Step 0:} Set $k=1$.
			
			{\bf Step 1:} Compute $x_{k+1} = \pr_K(x_k - \lambda_k F(x_k))$.
			
			{\bf Step 2:} Check $x_{k+1} = x_k$. {\bf If} Yes {\bf then} Stop. {\bf Else} set $k = k+1$ and go to {\bf Step 1}.
		\end{quote}
	\end{Algorithm}
	
	\medskip
	In comparison with Algorithm \ref{old_algorithm}, the stepsizes in Algorithm \ref{new_algorithm} are varied and forming a non-summable diminishing sequence of positive real numbers. In addition, the stepsizes in Algorithm ~\ref{new_algorithm} can be determined without knowing the modulus of strong pseudomonotonicity.
	
	\medskip
	If $\vi(K, F)$ is solvable, we will prove the iterative sequence in Algorithm \ref{new_algorithm} converges to the unique solution of $\vi(K, F)$. First, we need the following lemma which is a special case of \cite[Lemma 1.5]{xu09}.
	
	\begin{Lemma}\label{lemma}
		Let $\{\eta_k\}$ be a positive sequence satisfying $\displaystyle \sum_{k = 1}^\infty \eta_k = \infty$ and $\displaystyle \lim_{k\to \infty} \eta_k = 0$, $\{\delta_k\}$ be a real sequence satisfying $\displaystyle \lim_{k\to \infty} \delta_k = 0$. Assume that $\{a_k\}$ is a non-negative sequence such that $$a_{k+1} \le (1-\eta_k)a_k+\eta_k \delta_k, \quad \forall k\ge 1.$$
		Then $\{a_k\}$ converges to $0$.
	\end{Lemma}
	
	We are ready to prove the convergence of the iterative sequence in Algorithm \ref{new_algorithm}. 
	
	\begin{Theorem}\label{new-grad}
		Let $K\subset H$ be a non-empty closed convex set, $F: K\to H$ be a strongly pseudomonotone with modulus $\gamma$. Suppose that $\vi(K, F)$ is solvable and its unique solution is $x^\ast$. Then
		\begin{itemize}
			\item[(i)] every sequence $\{x_k\}$ produced by Algorithm \ref{new_algorithm} satisfies
			\begin{equation}
			\label{ineq:new_grad} \|x^\ast - x_{k+1}\|^2 \le (1-2\lambda_k \gamma) \|x^\ast - x_k\|^2+\lambda_k^2 \|F(x_k)\|^2.
			\end{equation}
			
			\item[(ii)] Suppose in addition that $F$ is bounded on $K$. Then $\{x_k\}$ converges in norm to $x^\ast$.
		\end{itemize}
	\end{Theorem} 
	
	\begin{Proof} (i) Since $\vi(K, F)$ is solvable and $F$ is strongly pseudomonotone, $\vi(K, F)$ admits a unique solution. Since $x^\ast$ is the solution of $\vi(K, F)$ and $x_k\in K$, we have $$\langle F(x^\ast), x_k - x^\ast \rangle \ge 0.$$
		This inequality and the strong pseudomonotonicity of $F$ imply that $$\langle F(x_k), x_k - x^\ast \rangle \ge \gamma \|x^\ast - x_k\|^2.$$
		Multiplying $\lambda_k$ to both sides, the latter inequality is equivalent to 
		\begin{equation}
		\label{grad2} \langle x^\ast, \lambda_k F(x_k) \rangle \le  \lambda_k\langle F(x_k), x_k \rangle - \lambda_k \gamma \|x^\ast - x_k\|^2.
		\end{equation}
		Since $x_{k+1} = \pr_K(x_k - \lambda_k F(x_k))$, it follows from Theorem~\ref{projection_char}(a) that $$\langle x^\ast - x_{k+1}, x_k - \lambda_k F(x_k) - x_{k+1} \rangle \le 0,$$
		which is equivalent to 
		\begin{equation}
		\label{grad1} \langle x^\ast, x_k - \lambda_k F(x_k) - x_{k+1} \rangle \le \langle x_{k+1}, x_k - \lambda_k F(x_k) - x_{k+1} \rangle.
		\end{equation}	
		Adding  \eqref{grad2} and \eqref{grad1}, we obtain $$\langle x^\ast, x_k - x_{k+1} \rangle \le  \langle x_{k+1}, x_k - \lambda_k F(x_k) - x_{k+1} \rangle + \lambda_k\langle F(x_k), x_k \rangle - \lambda_k \gamma \|x^\ast - x_k\|^2. $$
		This inequality can be written as
		\begin{equation}
		\label{grad3} 2\lambda_k \gamma \|x^\ast - x_k\|^2 \le 2\langle x_{k+1}- x^\ast, x_k - x_{k+1} \rangle + 2\lambda_k \langle F(x_k), x_k - x_{k+1} \rangle.
		\end{equation}
		Since
		$$2\langle x_{k+1}- x^\ast, x_k - x_{k+1} \rangle = \|x^\ast - x_k\|^2 - \|x^\ast - x_{k+1}\|^2 - \|x_k - x_{k+1}\|^2 \quad \text{and}$$
		$$
		2\lambda_k \langle F(x_k), x_k - x_{k+1}\rangle\leq \lambda_k^2 \|F(x_k)\|^2+\|x_k - x_{k+1}\|^2,
		$$
		it follows from \eqref{grad3} that 
		$$
		2\lambda_k \gamma \|x^\ast - x_k\|^2 \le  \|x^\ast - x_k\|^2 - \|x^\ast - x_{k+1}\|^2+\lambda_k^2 \|F(x_k)\|^2,
		$$
		which is equivalent to 
		$$
		\|x^\ast - x_{k+1}\|^2 \le (1-2\lambda_k \gamma) \|x^\ast - x_k\|^2+\lambda_k^2 \|F(x_k)\|^2.
		$$
		
		(ii) Let $a_k = \|x^\ast - x_k\|^2, \eta_k = 2 \lambda_k \gamma$ and $\delta_k = \frac{\|F(x_k)\|^2}{2\gamma} \lambda_k$. The sequence $\{\eta_k\}$ is a positive sequence satisfying $\displaystyle \sum_{k=1}^\infty \eta_k = \infty$ and $\displaystyle \lim_{k\to \infty} \eta_k = 0$. Since $F$ is bounded on $K$ and $\{x_k\}\subset K$, the sequence  $\{\|F(x_k)\|^2\}$ is bounded and so $\{\delta_k\}$ is a real sequence satisfying $\displaystyle \lim_{k\to \infty} \delta_k = 0$. Therefore, it follows from inequality \eqref{ineq:new_grad} and Lemma \ref{lemma} that $\|x^\ast - x_k\|^2 \to 0$, which implies $x_k \to x^\ast$. \qed
		
	\end{Proof}
	
	\begin{Remark}\label{remark_boundedness}
		The boundedness of $F$ over $K$ can be weakened by the following condition: there exists a point $x\in K$ such that $F$ is bounded on $K^\prime = K\cap \overline{B\left(x, \frac{1}{\gamma} \|F(x)\|\right)}$, where $\overline{B\left(x, \frac{1}{\gamma} \|F(x)\|\right)}$ is the closed ball with center $x$ and radius $\frac{1}{\gamma} \|F(x)\|$. If this condition is satisfied, by inequality \eqref{new_errorbound_xinK} in Remark \ref{remark_new_errorbound}, $K^\prime$ is a non-empty, closed, convex set containing $x^\ast$ and $F$ is strongly pseudomonotone, bounded on $K^\prime$. Moreover, $\vi(K, F)$ and $\vi(K^\prime, F)$ admits the same unique solution $x^\ast$. Thus, we can apply Algorithm \ref{new_algorithm} for $\vi(K^\prime, F)$ and the convergence of the iterative sequence to $x^\ast$ is ensured by Theorem \ref{new-grad}. The trade-off here is we have to project on $K^\prime$, which is more costly than projecting on the original $K$. 
	\end{Remark}
	
	When $H$ is finite dimensional, the boundedness of $F$ on $K$ can be replaced by the boundedness of $K$ and the continuity of $F$ on $K$. In that case, $\vi(K, F)$ is always solvable.
	
	\begin{Corollary}\label{col_finite_dim}
		Suppose that $H$ is finite dimensional. Let $K\subset H$ be a non-empty closed bounded convex set, $F\colon K\to H$ be a continuous and strongly pseudomonotone operator on $K$. Then every sequence $\{x_k\}$ produced by Algorithm \ref{new_algorithm} converges to the unique solution of $\vi(K, F)$. 
	\end{Corollary}
	
	\medskip	
	The next example shows that the boundedness of $F$ on $K$ in Theorem \ref{new-grad} and the boundedness of $K$ in Corollary \ref{col_finite_dim} cannot be omitted.
	
	\begin{Example}\label{ex:not-bounded-on-K}
		Let $H = K = \R, F(x) = 2^{|x|}x$ and $\lambda_k =\frac{1}{k}$ for all $k\ge 1$. The operator $F$ is strongly monotone with modulus $1$ (thus strongly pseudomonotone with modulus $1$) but not bounded on $K$ and $\vi(K, F)$ has a unique solution $x^\ast = 0$. Moreover, the sequence $\{\lambda_k\}$ satisfying $\displaystyle \sum_{k = 1}^\infty \lambda_k = \infty$ and $\displaystyle \lim_{k\to \infty} \lambda_k = 0$.
		The iterative sequence $\{x_k\}$ in Algorithm \ref{new_algorithm} is defined as $$x_{k+1} = x_k\left( 1-\frac{2^{|x_k|}}{k}\right), \quad \forall k\ge 1.$$
		Let $x_1 = 2$. We will prove by induction that $$|x_k| \ge 2k, \quad \forall k\ge 1.$$
		The inequality is true for $k = 1$. Assume that $|x_k|\ge 2k$, we have $$\begin{aligned} |x_{k+1}| &= |x_k| \left|1 - \frac{2^{|x_k|}}{k}\right| \\& = |x_k|\left(\frac{2^{|x_k|}}{k}-1\right) \\& \ge 2k\left(\frac{4^k}{k}-1\right)\\&\ge 2(k+1).\end{aligned}$$
		Hence $\{x_k\}$ is not bounded, which means $\{x_k\}$ does not converge to $x^\ast$.
	\end{Example}
	
	\section{Rate of convergence}
	
	In this section, we consider $\vi(K, F)$ when $K \subset H$ is a non-empty closed convex set, $F: K\to H$ is a strongly pseudomonotone operator with modulus $\gamma$ and bounded on $K$. Suppose that $\vi(K, F)$ is solvable. We will investigate the rate of convergence of Algorithm \ref{new_algorithm} when the stepsizes are sequences of terms defining the $p$-series, i.e, $$\lambda_k = \frac{1}{k^p}, \quad \text{where } p\in (0, 1].$$
	
	First, let us note that we can always scale the given operator $F$ by $\frac{1}{2\gamma}$ so that the resulting operator $F^\prime$ is strongly pseudomonotone with modulus $\gamma^\prime = \frac{1}{2}$ and $\vi(K, F^\prime)$ admits the same solution with $\vi(K, F)$. Thus, we only need to consider a strongly pseudomonotone operator with modulus $\frac{1}{2}$.
	
	Let $\{x_k\}$ be the iterative sequence generated by Algorithm \ref{new_algorithm}. Recall inequality \eqref{ineq:new_grad} (remind that we assumed $\gamma = \frac{1}{2}$):
	$$\|x^\ast - x_{k+1}\|^2\le \left(1 - \lambda_k \right)\|x^\ast - x_k\|^2 +   \lambda_k^2\|F(x_k)\|^2, \quad \forall k\ge 1.$$
	Since $F$ is bounded on $K$, there exists $M>0$ such that 
	$$
	\|F(x)\|\leq M, \quad \forall x\in K.
	$$
	Since $\{x_k\}\subset K$, it follows that 
	$$\|x^\ast - x_{k+1}\|^2\le \left(1 - \lambda_k \right)\|x^\ast - x_k\|^2 +   M^2\lambda_k^2, \quad \forall k\ge 1.$$
	For simplicity, denote $\|x^\ast - x_k\|^2 = a_k$. The above inequality becomes 
	
	\begin{equation}\label{rate_ineq}
	a_{k+1} \le (1 - \lambda_k) a_k + M^2 \lambda_k^2, \quad \forall k\ge 1.
	\end{equation}
	
	This inequality plays an important role in determining the rate of convergence of the algorithm. We will consider three cases: $p = 1, p\in \left(\frac{1}{2}, 1 \right)$ and $p\in \left(0, \frac{1}{2} \right]$. Let us remind that $\{a_k\}$ converges to $0$ with rate $O(b_k)$, where $\{b_k\}$ is a sequence known to converge to $0$, if there exists a constant $K > 0$ such that $$|a_k| \le K |b_k|, \quad \text{for sufficiently large } k$$
	(see \cite[Definition 1.18]{burden}).
	
	\subsection{The case $p = 1$}
	
	\begin{Theorem}\label{rate_p=1}
		Let $\{x_k\}$ be the sequence generated by Algorithm \ref{new_algorithm} with stepsizes $\lambda_k = \frac{1}{k}$ for all $k\ge 1$ and $x^\ast$ be the solution of $\vi(K, F)$. Then $\|x^\ast - x_k\|$ converges to $0$ with rate $O\left( \sqrt{\frac{\ln k}{k}}\right)$.
	\end{Theorem}
	
	\begin{Proof}
		
		From inequality \eqref{rate_ineq}, we have $$a_{k+1} \le \left(1 - \frac{1}{k} \right) a_k + M^2 \cdot \frac{1}{k^2}, \quad \forall k\ge 1.$$
		By induction, we obtain
		$$a_{k+1} \le \frac{a_2}{k} + \frac{M^2}{k}\left( \frac{1}{2} + \frac{1}{3} + \cdots + \frac{1}{k}\right), \quad \forall k\ge 2.$$
		By the well-known inequality 
		$$
		\frac{1}{2} + \frac{1}{3} +\cdots + \frac{1}{k} < \ln k, \quad\forall k\geq 2,
		$$ 
		it follows that  $$a_{k+1} < \frac{a_2}{k} + M^2 \frac{\ln k}{k}, \quad \forall k\ge 2.$$
		Thus $$a_{k+1} < \left(a_2 + M^2\right)\frac{\ln k}{k} \le 2(a_2+M^2)\frac{\ln(k+1)}{k+1}, \quad \forall k\ge 3.$$
		Therefore, the sequence $\{a_k\}$ converges to $0$ with rate $O\left(\frac{\ln k}{k} \right)$. In other words, $\|x^\ast - x_k\|$ converges to $0$ with rate $O\left( \sqrt{\frac{\ln k}{k}}\right)$. \qed
		
	\end{Proof}
	
	\subsection{The case $p\in \left(\frac{1}{2}, 1\right)$}
	
	\begin{Theorem}\label{rate_1/2 < p < 1}
		Let $\{x_k\}$ be the sequence generated by Algorithm \ref{new_algorithm} with stepsizes $\lambda_k = \frac{1}{k^p}$ for all $k\ge 1$ where $p\in \left(\frac{1}{2}, 1\right)$ and $x^\ast$ be the solution of $\vi(K, F)$. Then $\|x^\ast - x_k\|$ converges to $0$ with rate $O\left(k^{\frac{1}{2} - p}\right)$.
	\end{Theorem}
	
	\begin{Proof}
		From inequality \eqref{rate_ineq},  for every $k\geq 1$, we have 
		$$a_{k+1} \le \left(1 - \frac{1}{k^p} \right) a_k + M^2 \cdot \frac{1}{k^{2p}}\le \left(1 - \frac{1}{k} \right) a_k + M^2 \cdot \frac{1}{k^{2p}}.$$
		Following the proof of Theorem \ref{rate_p=1}, we get $$a_{k+1}\le \frac{a_2}{k} + \frac{M^2}{k} \left(\frac{1}{2^{2p-1}} + \frac{1}{3^{2p-1}} + \cdots + \frac{1}{k^{2p-1}} \right), \quad \forall k\ge 2.$$
		Since $1-2p < 0$, the function $x^{1-2p}$ is decreasing on $[1, \infty)$. It follows that $$\frac{1}{2^{2p-1}} + \frac{1}{3^{2p-1}} + \cdots + \frac{1}{k^{2p-1}} < \int_1^k x^{1-2p} dx = \frac{k^{2-2p}-1}{2-2p},$$
		then $$a_{k+1} \le \frac{a_2}{k} + \frac{M^2}{2-2p} k^{1-2p} \le  \left(a_2 + \frac{M^2}{2-2p} \right)k^{1-2p}\le 2\left(a_2 + \frac{M^2}{2-2p} \right)(k+1)^{1-2p}, \quad \forall k\ge 1.$$
		This implies the conclusion of the theorem. \qed
	\end{Proof}
	
	\subsection{The case $p\in \left(0, \frac{1}{2}\right]$}
	
	\begin{Theorem}
		Let $\{x_k\}$ be the sequence generated by Algorithm \ref{new_algorithm} with stepsizes $\lambda_k = \frac{1}{k^p}$ for all $k\ge 1$ where $p\in \left(0, \frac{1}{2}\right]$ and $x^\ast$ be the solution of $\vi(K, F)$. Then $\|x^\ast - x_k\|$ converges to $0$ with rate $O\left(k^{-\frac{p}{2}}\right)$.
	\end{Theorem}
	
	\begin{Proof}
		From inequality \eqref{rate_ineq}, we have
		\begin{equation} \label{rate_ineq'_p<1/2}
		a_{k+1} \le \left(1 - \frac{1}{k^p} \right) a_k + M^2 \frac{1}{k^{2p}}, \quad \forall k\ge 1.
		\end{equation}
		Let $\{u_k\}$ be defined recursively as $$u_2 = \frac{1}{2^p},\, u_k = \frac{k^p-1}{(k-1)^p}u_{k-1}+\frac{1}{k^p} , \quad \forall k\ge 3.$$
		Firstly, we prove by induction that 
		\begin{equation}\label{rate_ineq_p<1/2}
		a_{k+1} \le \left(1 - \frac{1}{k^p} \right)\left(1 - \frac{1}{(k-1)^p} \right) \cdots \left(1-\frac{1}{2^p}\right) a_2 + \frac{M^2}{k^p} u_k, \quad \forall k\ge 2.
		\end{equation}
		If $k = 2$, \eqref{rate_ineq_p<1/2} becomes $$a_3 \le \left(1 - \frac{1}{2^p} \right)a_2 + \frac{M^2}{2^p} u_2,$$
		which is true by \eqref{rate_ineq'_p<1/2}. 
		Suppose \eqref{rate_ineq_p<1/2} is true for $k$. By \eqref{rate_ineq'_p<1/2} and induction hypothesis, we have $$\begin{aligned} a_{k+2} &\le \left[1 - \frac{1}{(k+1)^p} \right] a_{k+1} + \frac{M^2}{(k+1)^{2p}}\\& \le \left[1 - \frac{1}{(k+1)^p} \right] \left[\left(1 - \frac{1}{k^p} \right)\left(1 - \frac{1}{(k-1)^p} \right) \cdots \left(1-\frac{1}{2^p}\right) a_2 + \frac{M^2}{k^p}u_k \right] +  \frac{M^2}{(k+1)^{2p}} \\& = \left[1 - \frac{1}{(k+1)^p} \right]\left(1 - \frac{1}{k^p} \right)\left[1 - \frac{1}{(k-1)^p} \right] \cdots \left(1-\frac{1}{2^p}\right) a_2 + \frac{M^2}{(k+1)^p}\left[\frac{(k+1)^p-1}{k^p}u_k+\frac{1}{(k+1)^p} \right] \\& = \left[1 - \frac{1}{(k+1)^p} \right]\left(1 - \frac{1}{k^p} \right)\left[1 - \frac{1}{(k-1)^p} \right] \cdots \left(1-\frac{1}{2^p}\right) a_2 + \frac{M^2}{(k+1)^p}u_{k+1}.\end{aligned}$$
		By induction principle, \eqref{rate_ineq_p<1/2} is true for all $k\ge 2$.
		
		\medskip\noindent
		Secondly, we  show that $\displaystyle \lim_{k\to \infty} u_k = 1$. By direct calculations, we have $$u_4 = 1+(4^p-1)\left(\frac{1}{9^p} - \frac{1}{12^p}\right) > 1.$$
		If $u_{k-1} > 1$ then $$u_k = \frac{k^p-1}{(k-1)^p}u_{k-1}+\frac{1}{k^p} > \frac{k^p-1}{(k-1)^p}+\frac{1}{k^p} > \frac{k^p-1}{k^p}+\frac{1}{k^p} = 1,$$
		thus $u_k > 1$ for all $k\ge 4$.
		Denote $v_k = u_k - 1$, then $\{v_k\}_{k\ge 4}$ is a positive sequence and $$v_{k} =  \frac{k^p-1}{(k-1)^p} v_{k-1}  + \frac{1}{k^p} + \frac{k^p-1}{(k-1)^p} - 1, \quad \forall k\ge 5.$$
		Let $$\eta_k = 1 - \frac{(k+1)^p-1}{k^p} \quad \text{and} \quad \delta_k = \left(\frac{1}{(k+1)^p} + \frac{(k+1)^p-1}{k^p} - 1\right)\frac{1}{1 - \frac{(k+1)^p-1}{k^p}}, \quad \forall k\ge 4,$$
		then $$v_{k+1} = (1-\eta_k)v_k + \eta_k\delta_k, \quad \forall k\ge 4.$$
		It is clear that $\displaystyle \lim_{k\to \infty} \eta_k = 0$. We will prove $\{\eta_k\}$ is a positive sequence. By Lagrange theorem, there exists $c_k\in (k, k+1)$ such that $$\eta_k = \frac{k^p-(k+1)^p+1}{k^p} = \frac{1 - pc_k^{p-1}}{k^p} > \frac{1 - c_k^{p-1}}{k^p} > 0, \quad \forall k\ge 4.$$  
		Next, we will prove that $\displaystyle \lim_{k\to \infty} \delta_k = 0$. We have $$\delta_k = \frac{k^p}{(k+1)^p(k^p - (k+1)^p+1)} - 1 = \frac{k^p}{(k+1)^p} \cdot \frac{1}{k^p - (k+1)^p + 1}  - 1.$$
		Since $\displaystyle \lim_{k\to \infty}\frac{k^p}{(k+1)^p} = 1$, we only need to prove $$\lim_{k\to \infty} \frac{1}{k^p - (k+1)^p + 1} = 1, \quad \text{or} \quad \lim_{k\to \infty}\left(k^p - (k+1)^p\right) = 0.$$
		By Lagrange theorem, for all $k$, there exists $c_k\in (k, k+1)$ such that $$k^p - (k+1)^p = -pc_k^{p-1}.$$
		When $k$ tends to $\infty$, $c_k$ also tends to $\infty$. Since $p < 1$, it follows $$|k^p - (k+1)^p| = pc_k^{p-1} \to 0 \quad \text{as} \quad k\to \infty.$$
		Therefore, $\displaystyle \lim_{k\to \infty} \delta_k = 0$.
		
		\medskip\noindent
		We continue to prove that $$\eta_k = 1 - \frac{(k+1)^p-1}{k^p} > \frac{1}{(k+1)^{2p}}, \quad \forall k\ge 4.$$
		This inequality is equivalent to $$\frac{(k+1)^{2p}-1}{(k+1)^{2p}} > \frac{(k+1)^p - 1}{k^p},$$
		or $$\frac{(k+1)^p+1}{(k+1)^{2p}} > \frac{1}{k^p}.$$
		We rewrite the above inequality as $$k^{-p} - (k+1)^{-p} < \frac{1}{(k+1)^{2p}}.$$
		By Lagrange theorem, there exists $c_k \in (k, k+1)$ such that $$k^{-p} - (k+1)^{-p} = \frac{p}{c_k^{p+1}}.$$
		Since $p+1 > 2p$ and $c_k > k$, for all $k\ge 4$ we have $$\begin{aligned} \frac{p}{c_k^{p+1}} &< \frac{1}{k^{p+1}} = \left(\frac{k+1}{k}\right)^{2p} \cdot \frac{1}{k^{1-p}} \cdot \frac{1}{(k+1)^{2p}} \\&< 2^{2p}\cdot \frac{1}{4^{1-p}}\cdot \frac{1}{(k+1)^{2p}} = \frac{1}{2^{2-4p}}\cdot \frac{1}{(k+1)^{2p}} < \frac{1}{(k+1)^{2p}}. \end{aligned}$$
		This leads to our desired inequality. Since $p < \frac{1}{2}$, it follows that $\displaystyle \sum_{k=4}^\infty \eta_k = \infty$.
		
		\medskip\noindent
		By Lemma \ref{lemma}, we have $\displaystyle \lim_{k\to \infty} v_k = 0$. Thus $\displaystyle \lim_{k\to \infty} u_k = 1$ which means $\{u_k\}$ is bounded above by some $C > 0$.
		
		\medskip\noindent
		Finally, following \eqref{rate_ineq_p<1/2}, 
		for all $k\geq 2$, we have 
		$$\begin{aligned} a_{k+1} &\le \left(1 - \frac{1}{k^p} \right)\left(1 - \frac{1}{(k-1)^p} \right) \cdots \left(1-\frac{1}{2^p}\right) a_2 + \frac{M^2}{k^p} u_k \\&\le \left(1 - \frac{1}{k} \right) \left( 1 - \frac{1}{k-1} \right) \cdots \left(1 - \frac{1}{2} \right) a_2 + \frac{M^2 C}{k^p} \\& = \frac{a_2}{k} + \frac{M^2 C}{k^p} \\&\le (a_2 + M^2 C) \frac{1}{k^p} \\& < 2(a_2 + M^2 C) \frac{1}{(k+1)^p}. \end{aligned}$$
		Therefore $\|x^\ast - x^k\|$ converges to $0$ with rate $O\left(k^{-\frac{p}{2}}\right)$. \qed
	\end{Proof}

	\section{Numerical experiments and comparison with related works}
	
	In this section, we will run some numerical experiments and compare our results with a recent work in  \cite{hai}. We recall the main algorithm  and its convergence results in 
	\cite{hai}.  
	
	\begin{Algorithm} {(see \cite[Algorithm 3.1]{hai})}
		\label{algo:hai}
		\begin{quote}
			{\bf Data.} Select $x_1\in K$ and a non-increasing sequence $\{\lambda_k\} \subset (0, \infty)$ satisfying $\lambda_k \to 0$ and $\displaystyle \sum_{k=1}^\infty \lambda_k = \infty$.
			
			{\bf Step 0:} Set $k = 1$.
			
			{\bf Step 1:} Compute $x_{k+1} = \pr_K\left(x_k - \frac{\lambda_k}{\max\{1, \|F(x_k)\|^2\}} F(x_k)\right)$.
			
			{\bf Step 2:} Check $x_{k+1} = x_k$. {\bf If} Yes {\bf then} Stop. {\bf Else} set $k = k+1$ and go to {\bf Step 1}.
		\end{quote}
	\end{Algorithm}
	\begin{Theorem} {\rm (see \cite[Theorem 3.1]{hai})}
		\label{thm:hai}
		Let $K\subset H$ be a non-empty closed convex set, $F: K\to H$ be strongly pseudomonotone and bounded on bounded subsets of $K$. Suppose that $\vi(K, F)$ is solvable. Then every sequence $\{x_k\}$ produced by Algorithm \ref{algo:hai} converges in norm to the unique solution of $\vi(K, F)$.
	\end{Theorem}
	
	We first compare the assumptions in Theorem~\ref{new-grad} and Theorem~\ref{thm:hai}.
	
	\begin{table}[ht]
		\centering
		\begin{tabular}{ |c| C{6cm} | C{6cm} | } 
			\hline
			\textbf{} & \bf{Theorem~\ref{new-grad}} & \bf{Theorem~\ref{thm:hai}} \\
			\hline
			Space & Hilbert & Hilbert \\ 
			\hline
			The set $K$ & Non-empty closed convex & Non-empty closed convex\\
			\hline
			The operator $F$ & Strongly pseudomonotone & Strongly pseudomonotone \\
			& Bounded on $K$ & Bounded on bounded subsets of $K$ \\
			\hline
			Stepsize & $\lambda_k$ &  $\frac{\lambda_k}{\max\{1, \|F(x_k)\|^2\}}$ \\
			\hline 	
		\end{tabular}
		\caption{\small Comparison in hypotheses of Theorem~\ref{new-grad} and Theorem~\ref{thm:hai}.}	
		\label{table:comparison_hypotheses}	
	\end{table}
	We can see that the main differences between Theorem~\ref{new-grad} and Theorem~\ref{thm:hai}  are the hypotheses on the operator $F$ and the choice of stepsizess. We will analyze each of these differences.
	\begin{enumerate}
		\item {\it About the hypothesis on the operator $F$:} although our hypothesis on $F$ in Theorem \ref{new-grad} is stronger than the hypothesis in Theorem~\ref{thm:hai}, we will give an in-depth analysis here. As discussed in Remark \ref{remark_boundedness}, the boundedness of $F$ in Theorem \ref{new-grad} can be replaced by a much weaker hypothesis, i.e. there exists a point $x\in K$ such that $F$ is bounded on $K^\prime = K\cap \overline{B\left(x, \frac{1}{\gamma} \|F(x)\|\right)}$, where $\overline{B\left(x, \frac{1}{\gamma} \|F(x)\|\right)}$ is the closed ball with center $x$ and radius $\frac{1}{\gamma} \|F(x)\|$. This hypothesis is much weaker than one in \cite[Theorem 3.1]{hai}, since we only need $F$ to be bounded on {\bf one} bounded subset of $K$, while \cite[Theorem 3.1]{hai} assumes $F$ to be bounded on {\bf every} bounded subset of $K$. However, projecting on $K^\prime$ may be practically more difficult than projecting on the original set $K$, so we keep such a ``strong" hypothesis in Theorem \ref{new-grad}.

		\item {\it About the stepsize:} our choice of stepsizes in Algorithm~\ref{new_algorithm} does not depend of the operator $F$. This leads to an advantage of our algorithm over Algorithm~\ref{algo:hai}: we can estimate the rate of convergence of the algorithm for each sequence of stepsizes. The convergence process of Algorithm~\ref{algo:hai}, on the other hand, may fluctuate depending on the operator $F$ as we will show below. 
	\end{enumerate}
	
	We test the convergence process of Algorithm \ref{new_algorithm} and Algorithm~\ref{algo:hai} in three different examples: \cite[Example 5.1]{hai}, \cite[Example 5.3]{hai} and Example \ref{ex:not-bounded-on-K}. The experiments are conducted in Python 3.7 with processor Intel(R) Core(TM) i7-1065G7 CPU @ 1.30GHz (8 CPUs). Let us formally recall \cite[Example 5.1]{hai}, \cite[Example 5.3]{hai} and Example \ref{ex:not-bounded-on-K}: 
	
	\begin{itemize}
		\item For \cite[Example 5.1]{hai}, the operator $F$ is $F(x) = Ax$ where $A$ is a randomly generated positive definite matrix and $K$ is the unit cube. In this experiment, we choose $A = M^T M + I$ where $M$ is an arbitrary matrix of which each entry is sampled from a standard Gaussian distribution and $I$ is the identity matrix; 
		
		\item For \cite[Example 5.3]{hai}, the operator $F$ is $F(x) = \left(\frac{1}{\|x\|} - \frac{1}{2}\right) x$ if $x\neq 0$ and $F(0) = 0$, while the set $K$ is the closed sphere with center $0$ and radius $\frac{1}{2}$;
		
		\item For Example \ref{ex:not-bounded-on-K}, the operator $F$ is $F(x) = 2^{\|x\|}x$. Following the same pattern in Example~\ref{ex:not-bounded-on-K}, we can prove that Algorithm \ref{new_algorithm} will not converge if $K = \R^n$. In this experiment, we choose $K$ to be the unit sphere.
	\end{itemize}
	
	In all examples, 
	\begin{itemize}
		\item we choose the sequence $\{\lambda_k\}$ to be $\frac{1}{k^p}$, where $0 < p \le 1$. In each case, we choose $p$ that yields the best performance for each algorithm. To be specific:
		\begin{itemize}
			\item[$\circ$] In \cite[Example 5.1]{hai}, we choose $p=1$ for Algorithm \ref{new_algorithm} and $p=0.1$ for Algorithm~\ref{algo:hai}.
			
			\item[$\circ$] In \cite[Example 5.3]{hai}, we choose $p=1$ for both algorithms.
			
			\item[$\circ$] In Example \ref{ex:not-bounded-on-K}, we choose $p=0.1$ for both algorithms.
		\end{itemize}
		
		\item $x_0$ is sampled from the standard Gaussian distribution.
	\end{itemize}
	
	The unique solution for $\vi(K, F)$ is $0$ in all three examples. We test the convergence processes of both algorithms in three different cases of $H$: $H = \R^{100}, H = \R^{500}$ and $H = \R^{1000}$. The results are presented in Figure \ref{fig:convergence_r100}, Figure \ref{fig:convergence_r500} and Figure \ref{fig:convergence_r1000}, respectively.

	Table \ref{table:time_iter} shows records of time and iterations of convergence processes of both algorithms in case $H = \R^{100}$ and $H = \R^{500}$. We do not report the result of the case $H = \R^{1000}$ in this table since Algorithm~\ref{algo:hai} converges too slowly in \cite[Example 5.1]{hai} (see the first sub-figure in Figure \ref{fig:convergence_r1000}). To be specific, Algorithm~\ref{algo:hai} takes more than 600,000 iterations to make $\|x_k\|$ less than $5$.
	
	\begin{table}[h!]
		\centering
		\begin{tabular}{cc|c|c||c|c||c|c|}
			\cline{3-8}
			&                                                 & \multicolumn{2}{c||}{\cite[Example 5.1]{hai}} & \multicolumn{2}{c||}{\cite[Example 5.3]{hai}} & \multicolumn{2}{c|}{Example \ref{ex:not-bounded-on-K}} \\ \cline{3-8} 
			&                                                 & Time                          & Iterations                        & Time                          & Iterations                        & Time                             & Iterations                           \\ \hline \hline
			\multicolumn{1}{|c|}{\multirow{2}{*}{$\R^{100}$}} & Algorithm \ref{new_algorithm} & \textbf{0.036s}                        & \textbf{221}                               & \textbf{0.004s}                        & 23                                & 0.004s                           & 28                                   \\ \cline{2-8} 
			\multicolumn{1}{|c|}{}                           & Algorithm~\ref{algo:hai}  & 0.79s                         & 3,523                              & 0.005s                        & 23                                & 0.004s                           & \textbf{22}                                 \\ \hline \hline
			\multicolumn{1}{|c|}{\multirow{2}{*}{$\R^{500}$}} & Algorithm \ref{new_algorithm} & \textbf{0.188s}                       & \textbf{947}                               & \textbf{0.005s}                        & 23                                & \textbf{0.004s}                            & 23                               \\ \cline{2-8} 
			\multicolumn{1}{|c|}{}                           & Algorithm~\ref{algo:hai}  & 39.2s                           & 127,010                          & 0.006s                        & 23                                & 0.005s                               & \textbf{22}                                   \\ \hline
		\end{tabular}
		\caption{\small Numerical experiments for both algorithms on different examples when $H = \R^{100}$ and $H = \R^{500}$. The iterative process stops when the average of $\|x_k\|$ in $20$ most recent iterations is less than $5\cdot 10^{-2}$. The results shown in the table are averages of $100$ distinct runs.}
		\label{table:time_iter}
	\end{table}

	While both algorithms act almost the same for \cite[Example 5.3]{hai} and Example \ref{ex:not-bounded-on-K}, there is a big difference in \cite[Example 5.1]{hai}. In specific, convergence processes of Algorithm \ref{new_algorithm} in three figures are similar, while Algorithm~\ref{algo:hai} sees a much slower convergence compared to Algorithm \ref{new_algorithm}.
	
	\begin{figure}[h!]
		\begin{multicols}{3}
			\includegraphics[width=6cm, height=4cm]{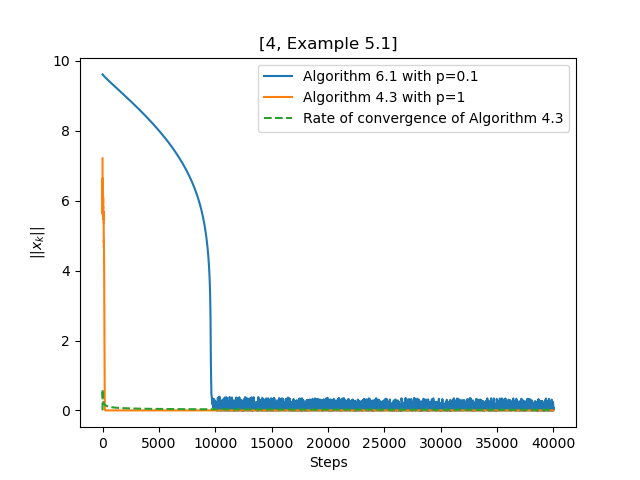} \par 
			\includegraphics[width=6cm, height=4cm]{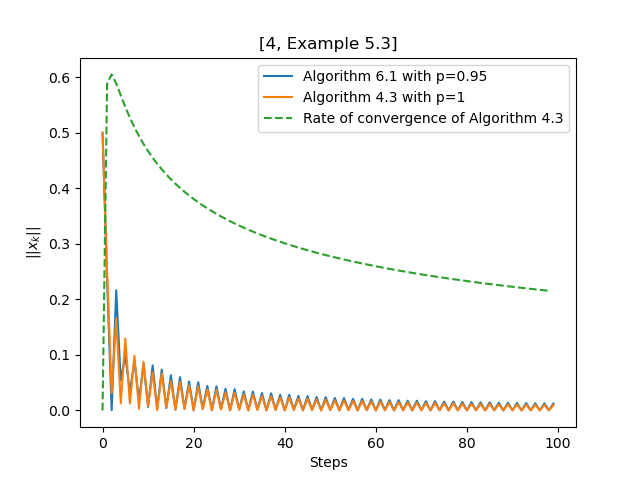} \par
			\includegraphics[width=6cm, height=4cm]{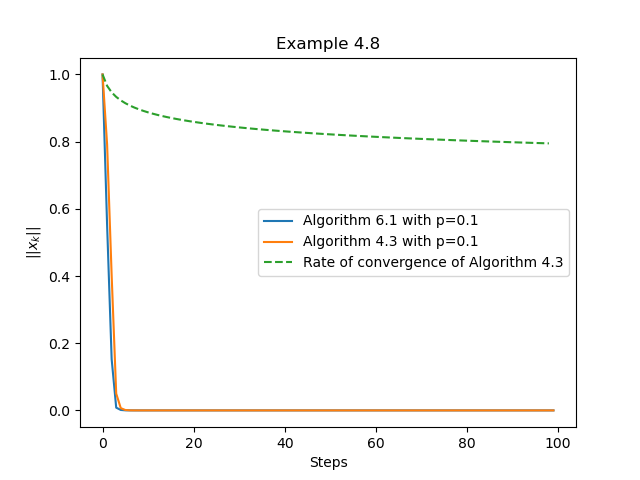} \par
		\end{multicols}
		\caption{\small Convergence of iterative sequences for different examples in $\R^{100}$. In \cite[Example 5.3]{hai}, for better visualization, we choose $p=0.95$ for Algorithm~\ref{algo:hai} and $p=1$ for Algorithm \ref{new_algorithm}. In fact, $p=1$ is the best choice for both algorithms. The same modification is applied for the case $H = \R^{500}$ and $H = \R^{1000}$.}
		\label{fig:convergence_r100}
	\end{figure}
	
	\begin{figure}[h!]
		\begin{multicols}{3}
			\includegraphics[width=6cm, height=4cm]{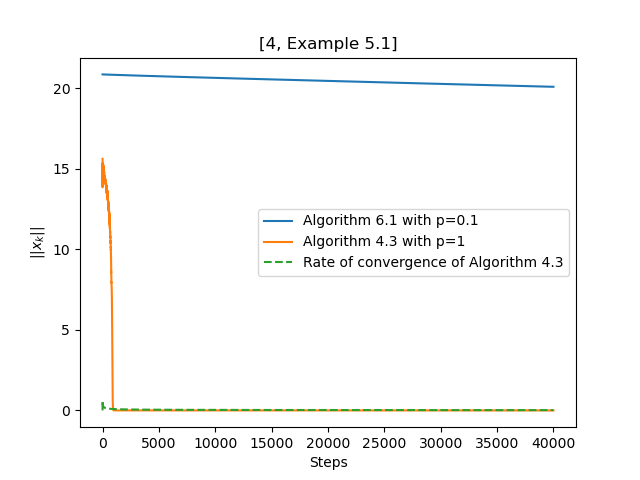} \par 
			\includegraphics[width=6cm, height=4cm]{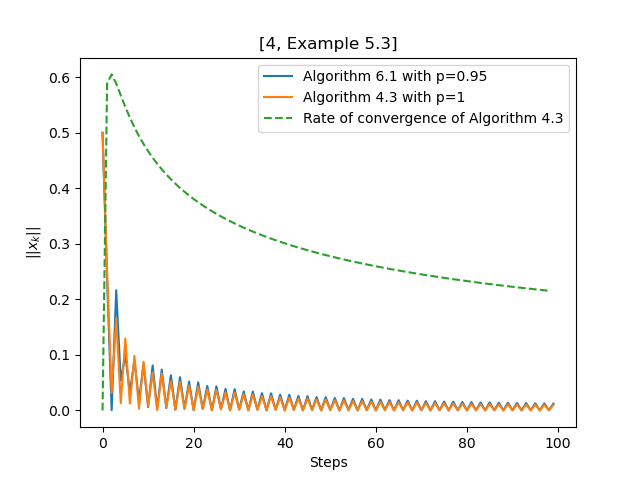} \par
			\includegraphics[width=6cm, height=4cm]{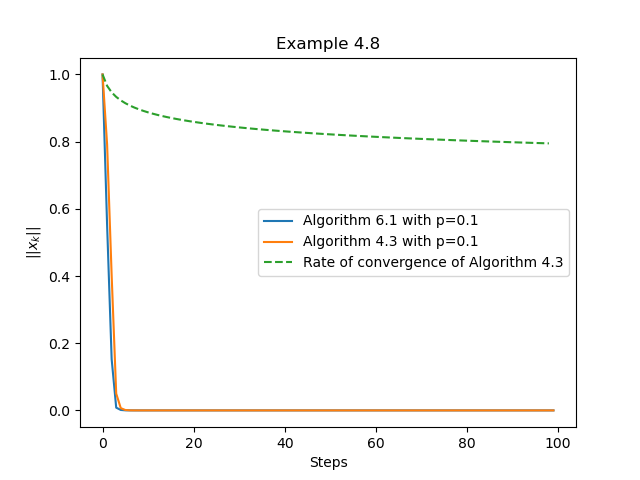} \par
		\end{multicols}
		\caption{\small Convergence of iterative sequences for different examples in $\R^{500}$.}
		\label{fig:convergence_r500}
	\end{figure}
	
	\begin{figure}[h!]
		\begin{multicols}{3}
			\includegraphics[width=6cm, height=4cm]{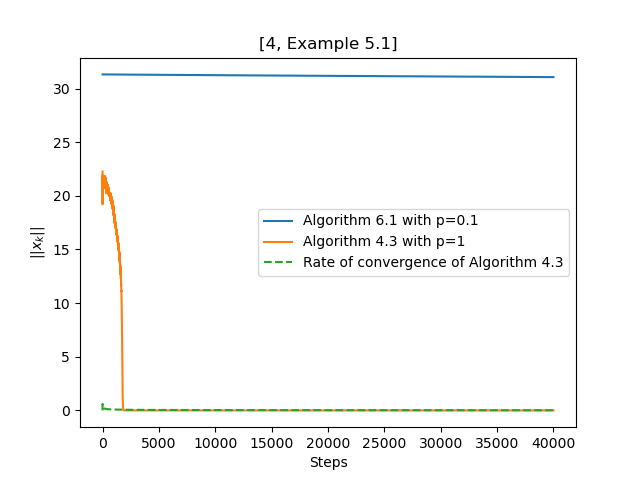} \par 
			\includegraphics[width=6cm, height=4cm]{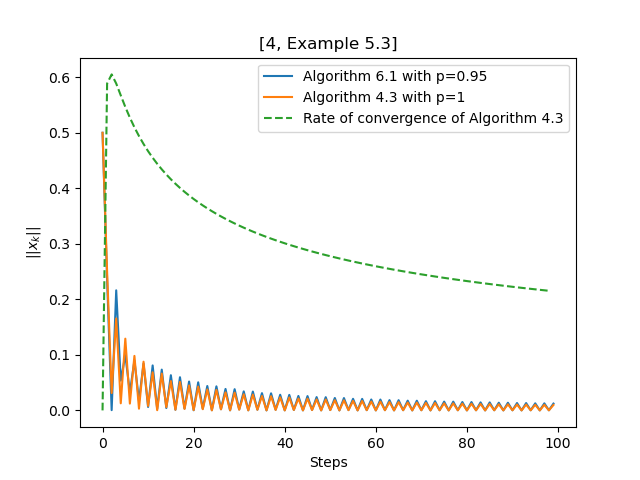} \par
			\includegraphics[width=6cm, height=4cm]{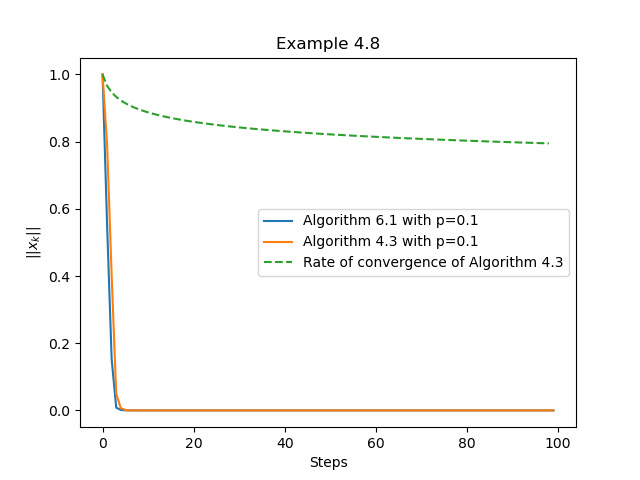} \par
		\end{multicols}
		\caption{\small Convergence of iterative sequences for different examples in $\R^{1000}$.}
		\label{fig:convergence_r1000}
	\end{figure}
	
	\begin{figure}[h!]
		\centering
		\begin{multicols}{2}
			\includegraphics[width=8cm, height=6cm]{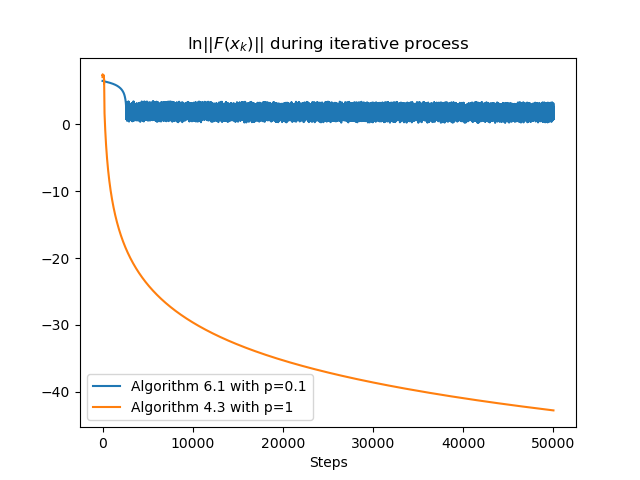}
			\includegraphics[width=8cm, height=6cm]{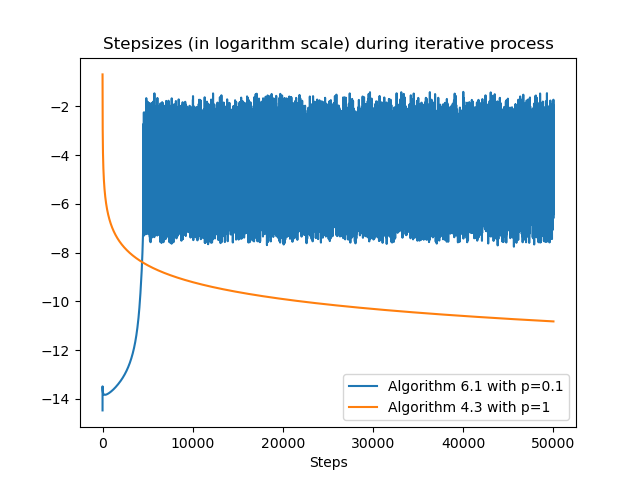}
		\end{multicols}
		
		\begin{multicols}{2}
			\includegraphics[width=8cm, height=6cm]{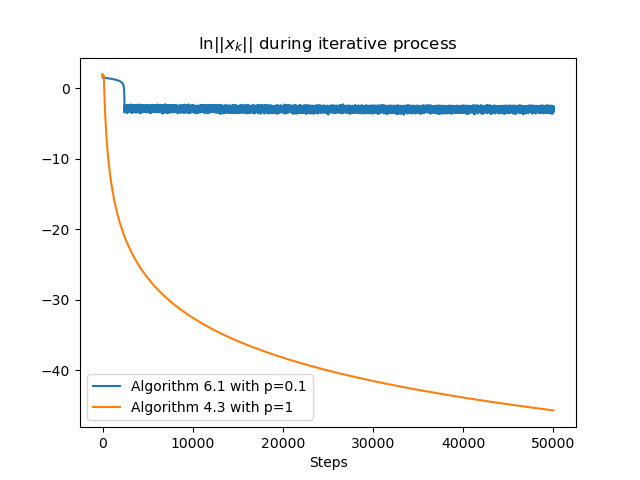}
			\includegraphics[width=8cm, height=6cm]{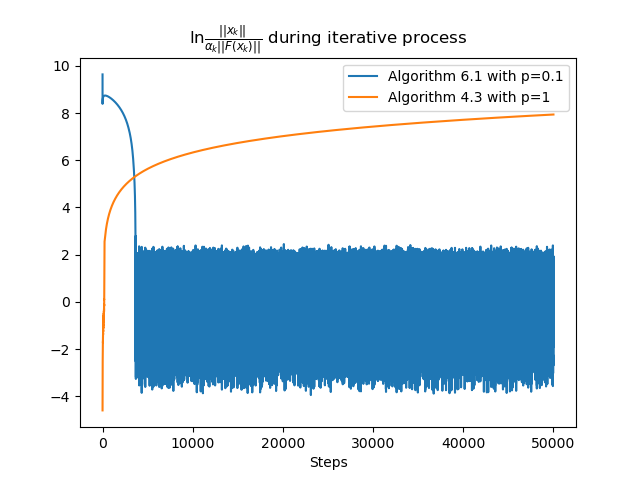}
		\end{multicols}
		\caption{\small Norm of $F(x_k)$, stepsizes sequence and the domination ratio of Algorithm \ref{new_algorithm} and Algorithm~\ref{algo:hai} during iterative process of \cite[Example 5.1]{hai} in case $H = \R^{100}$ (in the last subfigure, $\alpha_k$ is the stepsize at iteration $k$-th, i.e. $\alpha_k = \frac{1}{k}$ for Algorithm \ref{new_algorithm} and $\alpha_k = \frac{1/k^{0.1}}{\max\{1, \|F(x_k)\|^2\}}$ for Algorithm~\ref{algo:hai}).}
		\label{fig:norm_F(x_k)&stepsize}
	\end{figure}
	
	\begin{figure}[h!]
		\centering
		\includegraphics[width=7cm, height=5.5cm]{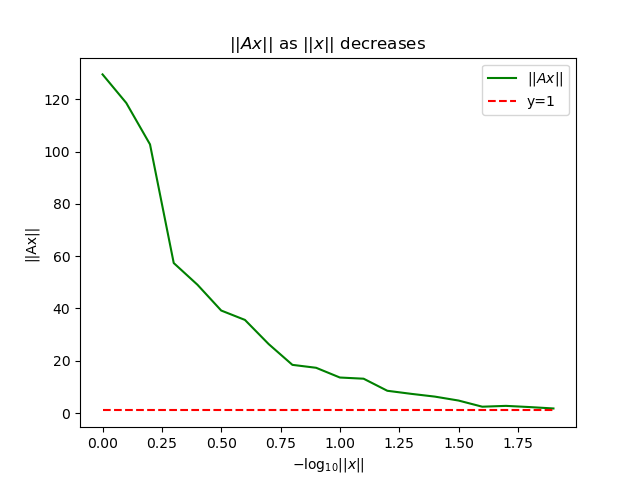}
		\caption{\small $\|Ax\|$ as $\|x\|$ decreases.}
		\label{fig:norm_Ax}
	\end{figure}	
	
	Let us formally explain this phenomenon when $H = \R^{100}$ via Figure \ref{fig:norm_F(x_k)&stepsize}. First, we would like to introduce some notations: we denote $\alpha_k^{(1)}$ as the stepsize of Algorithm \ref{new_algorithm} at step $k$, i.e. $\alpha_k^{(1)} = \frac{1}{k}$. Similarly, we denote $\alpha_k^{(2)}$ as the stepsize of Algorithm~\ref{algo:hai}, i.e. $\alpha_k^{(2)} = \frac{1/k^{0.1}}{\max\{1, \|F(x_k)\|^2\}}$. When it is not necessary to distinguish between the two, we denote $\alpha_k$ as the stepsize at step $k$ for either algorithm.
	
	At the beginning of the process, $\|F(x_k)\|$ is large, which means the stepsize $\alpha_k^{(2)}$ is very small. This makes the convergence of $\{x_k\}$ in Algorithm~\ref{algo:hai} slow. The sequence $\{x_k\}$ continues to decrease until it reaches somewhere around $10^{-2}$. This is where something interesting happens. Let us move our attention to the ratio $\frac{\|x_k\|}{\alpha_k \|F(x_k)\|}$, which we call the \textit{domination ratio}. As $k$ increases, the domination ratio for Algorithm \ref{new_algorithm} is 
	$$\frac{\|x_k\|}{\alpha_k^{(1)} \|F(x_k)\|} = \frac{k \|x_k\|}{\|Ax_k\|} \ge \frac{k}{\|A\|} \to \infty.$$
	
	This means in the latter part of the process, $x_k$ dominates the term $x_k - \alpha_k F(x_k)$. This helps the process of Algorithm \ref{new_algorithm} converge smoothly. On the other hand, the domination ratio of Algorithm~\ref{algo:hai} is 
	$$\frac{\|x_k\|}{\alpha_k^{(2)} \|F(x_k)\|} = k^{0.1} \|x_k\| \|F(x_k)\|.$$
	
	Note that when $\|x_k\|$ is close to (but still greater than) $10^{-2}$, $\|F(x_k)\|$ is close to $1$ (see Figure \ref{fig:norm_Ax}), and the value of $k$ is around $4000$. Thus 
	$$\frac{\|x_k\|}{\alpha_k^{(2)} \|F(x_k)\|} = k^{0.1} \|x_k\| \|F(x_k)\| \approx k^{0.1} \|x_k\| \approx 4000^{0.1} \cdot 10^{-2} \approx 0.02.$$
	
	This means $\alpha_k^{(2)} F(x_k)$ dominates the term $x_k - \alpha_k^{(2)} F(x_k)$! What happens next is clear in all four subfigures of Figure \ref{fig:norm_F(x_k)&stepsize}: everything fluctuates. That is the battleground for $x_k$ and $\alpha_k^{(2)} F(x_k)$ to gain domination in the term $x_k - \alpha_k^{(2)} F(x_k)$. This battle can be outlined as follow:
	\begin{enumerate}
		\item $\|x_k\|$ is close to $10^{-2}$, which makes
		\item the domination ratio small, which means
		\item $\alpha_k^{(2)} F(x_k)$ dominates, which means
		\item $x_{k+1}$ is heavily affected by $\alpha_k^{(2)} F(x_k)$, which means
		\item $\|x_{k+1}\|$ is large, which makes
		\item $\|F(x_{k+1})\|$ large, which makes
		\item the domination ratio large, which means
		\item $x_{k+1}$ dominates, which means
		\item $x_{k+2}$ is mainly affected by $x_{k+1}$, which means
		\item after some steps, say $s$ steps, $\|x_{k+s}\|$ is close to $10^{-2}$ again.
	\end{enumerate}
	This cycle makes the convergence process worse since then. 
	
	\begin{Remark}\label{remark:different_setting_K}
		In the above experiments, we run and report results of Example \ref{ex:not-bounded-on-K} when $K$ is the unit sphere. However, we know that Algorithm~\ref{algo:hai} can still converge when $K$ is the entire real $n$-dimensional space. In this remark, we will re-conduct the experiment for Example \ref{ex:not-bounded-on-K} with different settings for the set $K$. With Algorithm~\ref{algo:hai}, we take $K$ to be $\R^n$. With Algorithm \ref{new_algorithm}, we choose $K$ as discussed in Remark \ref{remark_boundedness}. First, we take an arbitrary $x_{\text{init}}$, say, the first unit vector in $\R^n$. By inequality \eqref{new_errorbound_xinK} in Remark \ref{new_errorbound}, we have 
		$$\|x^\ast - x_{\text{init}}\| \le \frac{1}{\gamma} \|F(x_{\text{init}})\| = 2^{\|x_{\text{init}}\|} \|x_{\text{init}}\| = 2,$$
		which means $x^\ast$ lies in the closed sphere with center $x_{\text{init}}$ and radius $2$. We take $K$ to be this set. Results of the experiment are reported in Figure \ref{fig:convergence_with_different_K}.
		
		\begin{figure}[h!]
			\begin{multicols}{3}
				\includegraphics[width=6cm, height=5cm]{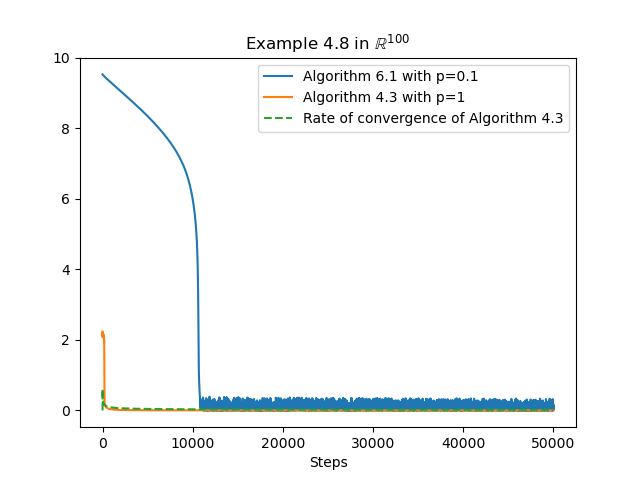} \par 
				\includegraphics[width=6cm, height=5cm]{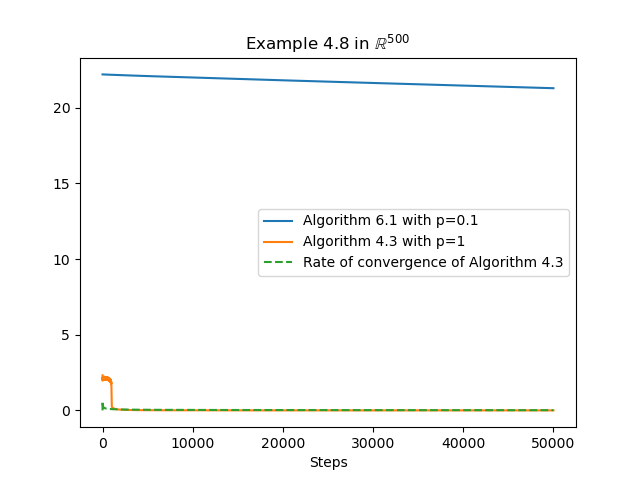} \par
				\includegraphics[width=6cm, height=5cm]{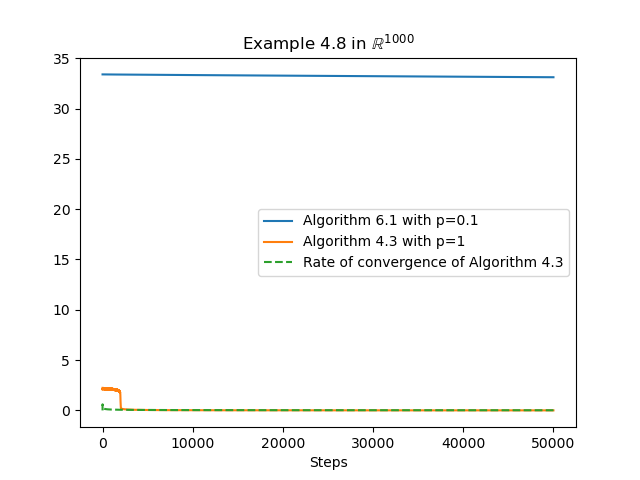} \par
			\end{multicols}
			\caption{\small Convergence of iterative sequences of the two algorithms in Example \ref{ex:not-bounded-on-K} with different settings on $K$.}
			\label{fig:convergence_with_different_K}
		\end{figure}
		
		As we can see, Algorithm~\ref{algo:hai} faces the same problem as pointed out in \cite[Example 5.1]{hai}. While Algorithm \ref{new_algorithm} converges fast as usual, Algorithm~\ref{algo:hai} suffers from the slow convergence at the beginning (due to the large denominator $\max\{1, \|F(x_k)\|^2\}$ of the stepsize) and the fluctuation phenomenon in the latter part of the process.
	\end{Remark}

	\section{Concluding remarks}
	
	In this article, we obtained an error bound and proved the convergence of iterative sequences generated by modified GPM for VIs governed by strongly pseudomonotone operators. Two counter-examples were given to show the necessity of Lipschitz continuity assumption in classical GPM as well as the boundedness hypothesis in modified GPM. Rate of convergence was estimated when the stepsizes are sequences of terms defining the $p$-series. We also conducted several numerical experiments and gave an in-depth comparison with a related algorithm.

	\medskip
	There are still some open questions for whom who may concern:
	\begin{enumerate}
		\item The extragradient projection method (EPM) (see \cite{korpelevich}) is another classical method solving a wider class of VIs than the GPM, i.e, VIs with monotone and Lipschitz continuous operators. In \cite{khanh2016}, Khanh proved that modified EPM with variable stepsizes is applicable for strongly pseudomonotone and Lipschitz continuous VIs. It is natural to ask whether modified EPM could solve VIs governed by pseudomonotone operators.
		
		\item It is also worth to consider the choice of $p$ to optimize the speed of convergence of iterative sequences produced by Algorithm \ref{new_algorithm} when $\lambda_k = \frac{1}{k^p}$ for all $k\ge 1$. Obviously, the optimized value of $p$ is not the same for all cases but depends on the constraint set $K$ and operator $F$. 
	\end{enumerate}
	
	{\small \noindent {\bf Acknowledgement.} We would like to thank Mr. Huynh Phuoc Truong for his comments and discussion in Example \ref{notLip}. We are also grateful to the anonymous referee and the associate editor for constructive comments and suggestions, which greatly improved the paper. Pham Duy Khanh was supported, in part, by the Fondecyt Postdoc Project 3180080, the Basal Program CMM--AFB 170001 from CONICYT--Chile, and  the National Foundation for Science and Technology Development (NAFOSTED) under grant number 101.01-2017.325.}

\end{document}